\documentclass[reqno,12pt,letterpaper]{amsart}
\usepackage{amsmath,amssymb,amsthm,graphicx,mathrsfs,url}
\usepackage[usenames,dvipsnames]{color}
\usepackage[colorlinks=true,linkcolor=Red,citecolor=Green]{hyperref}

\def\?[#1]{\textbf{[#1]}\marginpar{\Large{\textbf{??}}}}
\renewcommand{\Re}{\mathop{\rm Re}\nolimits}
\renewcommand{\Im}{\mathop{\rm Im}\nolimits}
\DeclareMathOperator{\Op}{Op}

\DeclareMathOperator{\supp}{supp}

\DeclareMathOperator{\Vol}{Vol}

\DeclareMathOperator{\dist}{dist}

\DeclareMathOperator{\WF}{WF}
\DeclareMathOperator{\Av}{Av}

\newtheorem{prop}{Proposition}
\newtheorem{thm}[prop]{Theorem}

\newtheorem{lem}[prop]{Lemma}

\newtheorem{rem}[prop]{Remark}

\numberwithin{equation}{section}
\numberwithin{prop}{section}

\begin{document}
\title{Semiclassical Cauchy Estimates and applications}
\author{Long Jin}
\email{jinlong@math.berkeley.edu}
\address{Department of Mathematics, Evans Hall, University of California,
Berkeley, CA 94720, USA}
\date{}
\maketitle

\section{Introduction}
In this note, we study the solutions to semiclassical Schr\"{o}dinger
equations on a real analytic manifold $ M $ of dimension $ n > 1$:
\begin{equation}\label{eq:sch}
(-h^2\Delta_g+V(x)-E(h))u(h)=0,
\end{equation}
where $V$ and $ g $ are  real-analytic function and metric on $ M $,
respectively and  $E(h)\to E_0$ as $h\to0$. We also consider more
general differential operators and,  when  $ M = {\mathbb R}^n  $,
analytic pseudodifferential operators satisfying suitable ellipticity condition.

The analyticity of $ V $ and of the metric $ g $ imply that solutions
are real analytic \cite[Theorem 8.6.1, 9.5.1(for hyperfunctions)]{H1}
and in particular Cauchy estimates hold:
\[  \sup_{ K} |\partial^\alpha u ( x ) | \leq C_h^{|\alpha| } | \alpha
|^{|\alpha | } , \ \ \forall \, K \Subset M  , \]
for some constant $C_h$ depending on $h$.
The semiclassical Cauchy estimate provides the following improvement:
\begin{equation}
\label{eq:semC}   \sup_{K} |\partial^\alpha u ( x ) | \leq h^{-\frac{n-1}{2}} C^{|\alpha|} ( h^{-1} + | \alpha |)^{|\alpha | } , \ \ \forall \, K \Subset M, \end{equation}
where $ C $ depends only on $ g$, $ V $ and $ K$.

The proof of \eqref{eq:semC} uses the FBI transform approach to analytic semiclassical theory developed by Sj\"ostrand \cite{Sj} and Martinez \cite{M}. It is presented in Section \ref{sec:sc} that near every point the solution $u$ can be analytic continued to a holomorphic function in a uniform complex neighborhood. Moreover, the analytic continuation will grow at most exponentially in $h^{-1}$, which as we will see, is equivalent to the semiclassical version of the Cauchy estimates on the derivatives of $u$.

We should remark that for differential operators one can obtain
estimates equivalent to \eqref{eq:semC} (see Proposition \ref{pr:equ}) by using H\"{o}rmander's approach to analytic hypoellipticity and rescaling -- see \cite[Lemma 7.1]{D-F}.  In fact, we learned about this after proving \eqref{eq:semC} directly using the FBI transform and the study of the Donnelly-Fefferman paper \cite{D-F} led to applications to the volume of nodal sets (zero set of $u(h)$) in the semiclassical setting.  An illustration of the level sets of eigenfunctions is shown in Fig.~\ref{f:1}: the zero sets occur in the regions where the eigenfunction is ``small'' and, in particular, are indistinguishable from the classically forbidden regions.

\begin{figure}[ht]
\includegraphics[width=3.15in]{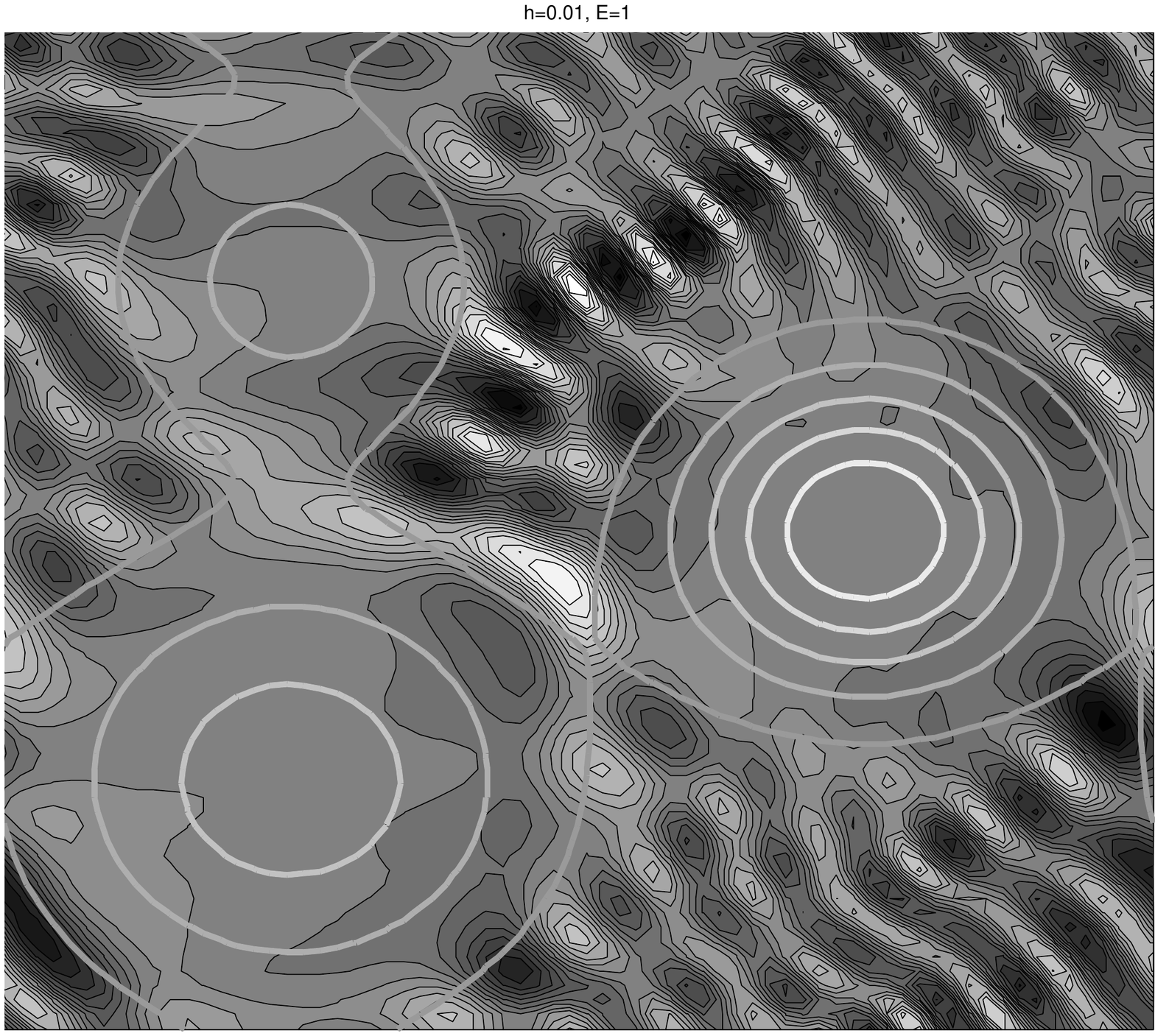} \includegraphics[width=3.15in]{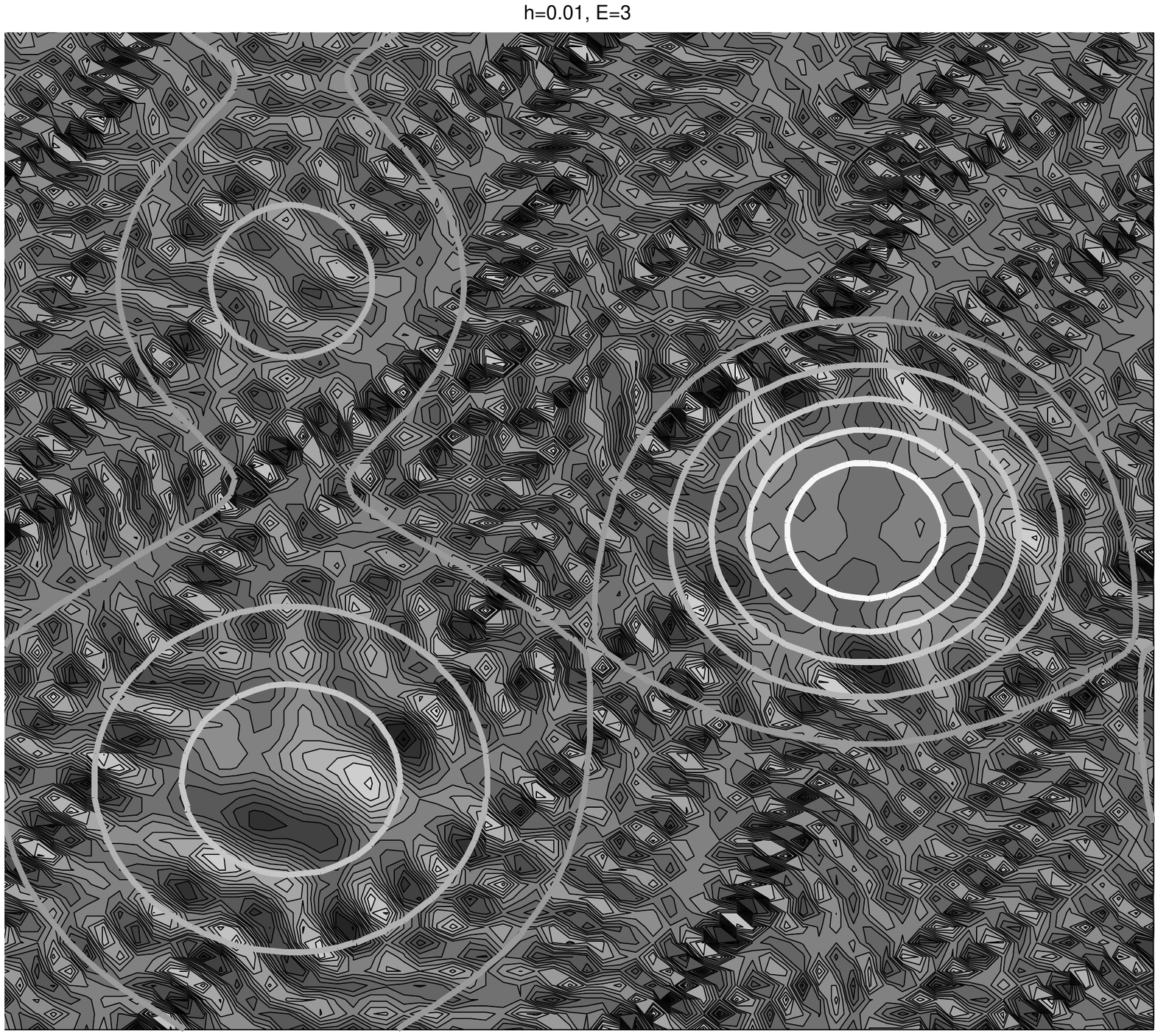}
\caption{Level sets of eigenfunctions of $ - h^2 \Delta + V $ on a
torus, $ [ 0 ,1 ] \times [ 0 , 1] $ where $ h = 0.01$ and $ V $ is a
 periodized sum of three bumps:  $   5e^{-10((x-0.75)^2 + (y-0.5)^2))}  +
 2e^{-10((x+0.25)^2 + (y-0.75)^2)} +
3e^{-5((x+0.25)^2 + (y+0.25)^2)} $ (with level sets shown). In the first picture the energy
 level is close to $ 1 $ and in the second, to $ 3 $ so that the
difference in classically forbidden regions is clearly visible.}
\label{f:1}
\end{figure}

Section \ref{sec:dp} contains a proof of the doubling property of solutions to \eqref{eq:sch} where we allow the manifold and the potential to be merely smooth. This type of results have been proved in more general setting, e.g. \cite{B-C} for $C^1$-potentials and are closely related to the unique continuation problems. There are two different ways to achieve such kind of results: the usual approach is through the Carleman-type estimates which establish a priori estimates with a weight; another approach was developed by Garofalo and Lin \cite{G-L} based on a combination of geometric and variational ideas. We shall follow the usual approach.

Finally in Section \ref{sec:ap}, we study the vanishing property of solutions to \eqref{eq:sch}. We shall show that the vanishing order of $u$ at a point is at most $Ch^{-1}$ and the nodal set of $u$, i.e. the set where $u$ vanishes, has $(n-1)$-dimensional Hausdorff measure $\sim h^{-1}$. When $V\equiv0$, i.e. $u$ is eigenfunctions of the Laplacian operator on $M$ with eigenvalues $Eh^{-2}$, this is the analytic case of Yau's conjecture \cite{Y} and is proved by Donnelly-Fefferman \cite{D-F}. We shall follow their argument closely. In the smooth setting, this is still an open problem, exponential types of upper and lower bounds were first established by Hardt and Simon \cite{Ha-Si}, see the notes \cite{H-L} for a detailed study on nodal sets and \cite{C-M}, \cite{S-Z}, \cite{He-So}, \cite{S-Z2} for recent progress on Yau's conjecture. Also see \cite{Z-Z} for nodal sets of semiclassical Schr\"{o}dinger operators in the smooth setting and \cite{B-H} for the physics perspective.

\begin{figure}[ht]
\includegraphics[width=3.15in]{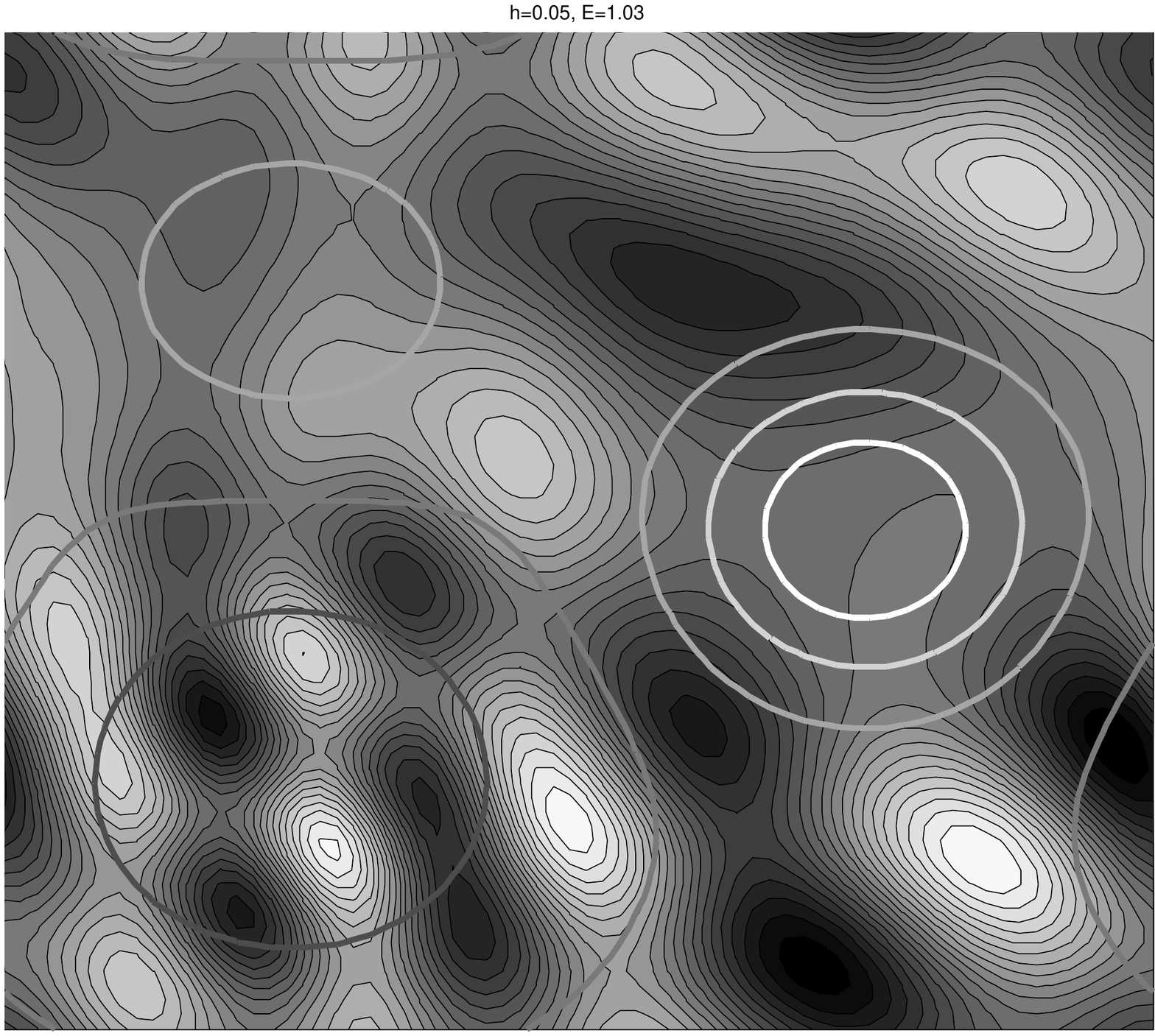} \includegraphics[width=3.15in]{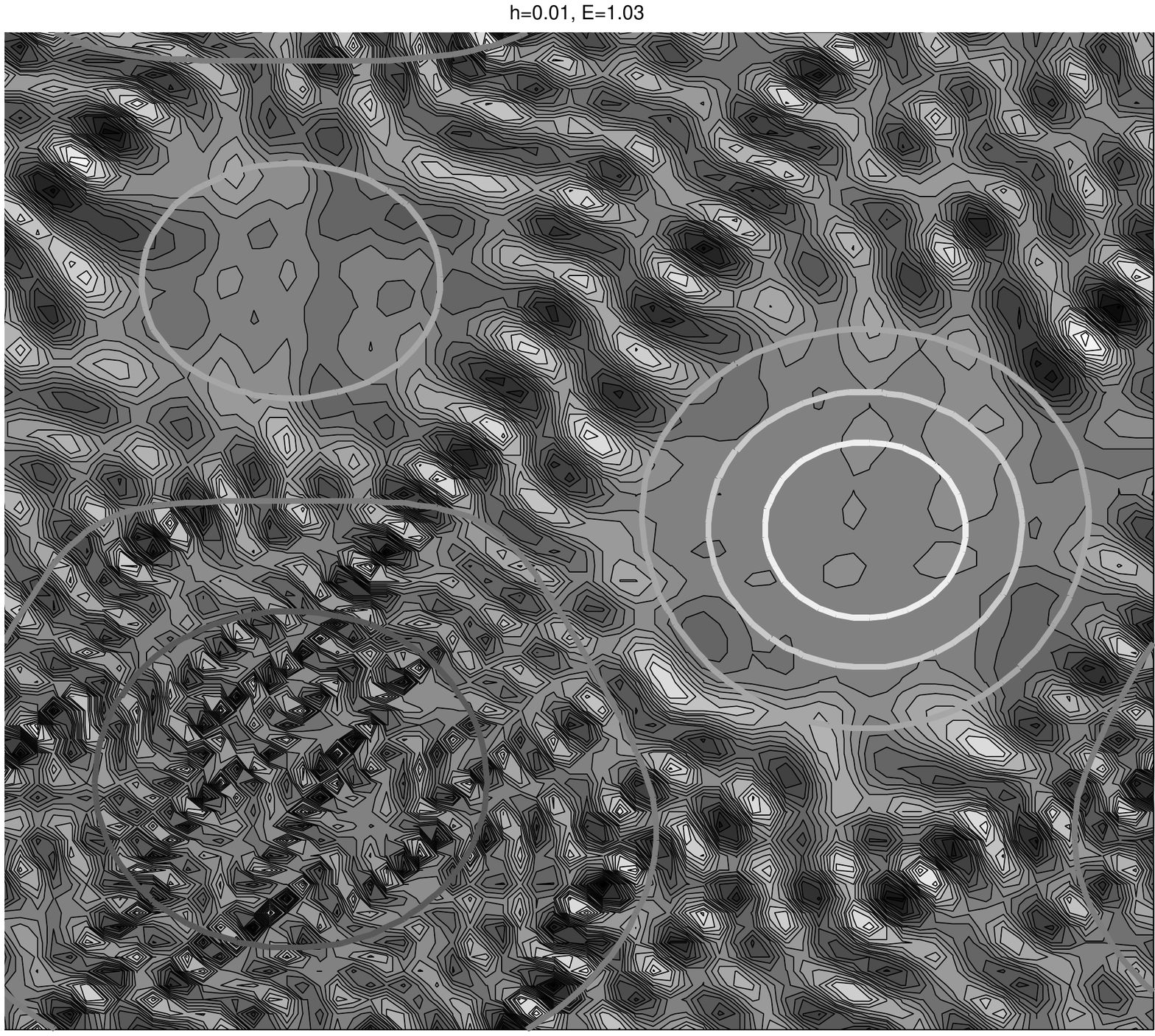}
\caption{Level sets of eigenfunctions of $ - h^2 \Delta + V $ on a
  torus, $ [0 ,1 ] \times [ 0 , 1] $ where $ V $  is a
 periodized sum of two bumps and one well :  $   5e^{-10((x-0.75)^2 + (y-0.5)^2))}  +
 2e^{-10((x+0.25)^2 + (y-0.75)^2)} -
3e^{-5((x+0.25)^2 + (y+0.25)^2)} $ (with level sets shown). The energy level is
now fixed at $ E = 1$ but $ h = 0.05 $ and $ h = 0.01$.}
\label{f:2}
\end{figure}

\subsection*{\normalsize Acknowledgement}
The author would like to thank Maciej Zworski for the encouragement
and advice during the preparation of this paper.  Thanks go also to
Chris Wong for providing a MATLAB code for calculating eigenfunctions for Schr\"odinger operators on tori, and to the National Science
Foundation for support in the Summer of 2012 under the grant
DMS-1201417.

\section{Semiclassical Cauchy Estimates and Analytic Continuation}
\label{sec:sc}

\subsection{Fourier-Bros-Iagolnitzer Transform}
\label{sec:fbi}

In this section we review some basic facts of Fourier-Bros-Iagolnitzer transform. For $h>0$, we define $T_h:\mathscr{S}'(\mathbb{R}^n)\to\mathscr{S}'(\mathbb{R}^{2n})$ as
\begin{equation}\label{eq:fbi}
T_hu(x,\xi)=2^{-\frac{n}{2}}(\pi h)^{-\frac{3n}{4}}\int e^{\frac{i}{h}(x-y)\xi-\frac{1}{2h}(x-y)^2}u(y)dy.
\end{equation}
In other words, $T_hu(x,\xi)=\langle\phi_{x,\xi},u\rangle_{\mathscr{S},\mathscr{S}'}$ where $\phi_{x,\xi}$ is the so-called coherent state centered at $(x,\xi)$, so $T_hu(x,\xi)$ captures the microlocal property of $u$ at $(x,\xi)\in\mathbb{R}^{2n}$. We state some basic properties of the FBI transform:

(1)  If $u\in\mathscr{S}'(\mathbb{R}^n)$, then $e^{\frac{\xi^2}{2h}}T_hu(x,\xi)$ is a holomorphic function of $z=x-i\xi\in\mathbb{C}^n$. In fact, $T_h(\mathscr{S}'(\mathbb{R}^n))=\mathscr{S}'(\mathbb{R}^{2n})\cap e^{-\frac{\xi^2}{2h}}\mathcal{H}(\mathbb{C}^n_{x-i\xi})$ where $\mathcal{H}(\mathbb{C}^n_{x-i\xi})$ is the space of entire functions on $\mathbb{C}^n$. This also shows $hD_xT_hu=(\xi+ihD_\xi)T_hu$.

(2) For every $u\in\mathscr{S}'(\mathbb{R}^n)$, $u=T_h^\ast T_hu$ where $T_h^\ast$ is defined as
\begin{equation*}
T_h^\ast v(y)=2^{-\frac{n}{2}}(\pi h)^{-\frac{3n}{4}}\int e^{-\frac{i}{h}(x-y)\xi-\frac{1}{2h}(x-y)^2}v(x,\xi)dxd\xi
\end{equation*}
(interpreted as an oscillatory integral with respect to $\xi$.)

(3) If $u\in L^2(\mathbb{R}^n)$, then $T_hu\in L^2(\mathbb{R}^{2n})$ and $\|T_hu\|_{L^2(\mathbb{R}^{2n})}=\|u\|_{L^2(\mathbb{R}^n)}$. Moreover, $T_hT_h^\ast$ is the orthogonal projection from $L^2(\mathbb{R}^{2n})$ onto $T_h(L^2(\mathbb{R}^n))=L^2(\mathbb{R}^{2n})\cap e^{-\frac{\xi^2}{2h}}\mathcal{H}(\mathbb{C}^n_{x-i\xi})$.

(4) Let $p\in S_{2n}(1)$, then $\tilde{p}(x,\xi,x^\ast,\xi^\ast)=p(x-\xi^\ast,x^\ast)$ belongs to $S_{4n}(1)$ and we have
\begin{equation*}
T_h\circ p(x,hD_x)=\tilde{p}(x,\xi,hD_x,hD_\xi)\circ T_h.
\end{equation*}
where $x^\ast$ and $\xi^\ast$ are the dual variables of $x$ and $\xi$ respectively. This formula is exact and $\tilde{p}$ does not depend on $\mu$ or which quantization we are using.

Bros and Iagolnitzer first use this type of transform to characterize analytic wavefront set: $(x_0,\xi_0)\in\WF_a(u)$ if and only if $T_hu(x,\xi)=O(e^{-\frac{c}{h}})$ uniformly in a neighborhood of $(x_0,\xi_0)$ for some $c>0$. In general, we can define FBI transform with a phase which ``looks like'' the standard phase above and an elliptic analytic symbol. All such FBI transform can be used to characterize analytic wavefront set, see \cite{D}, \cite{Sj} and \cite{Z} for this general approach. For convenience, we shall only consider the standard FBI transform and the following modification.

\begin{lem}[Change of FBI by an analytic symbol]\label{pr:chfbi}
Suppose $a=a(x,y,\xi)$ is an analytic symbol defined for $x,y\in\mathbb{C}^n,\xi\in\mathbb{R}^n$ and is of tempered growth in $x,y,\xi$,
$\mu>1$ fixed, then we define $T'_h$ as
\begin{equation}
T'_hu(x,\xi)=2^{-\frac{n}{2}}(\pi h)^{-\frac{3n}{4}}\int e^{\frac{i}{h}(x-y)\xi-\frac{\mu}{2h}(x-y)^2}a(x,y,\xi)u(y)dy
\end{equation}
We have if $|T_hu(x,\xi)|=O(e^{-c/h})$ in a real neighborhood $U$ of $(x_0,\xi_0)$, then $|T_h'u(x,\xi)|=O(e^{-c'/h})$ in a neighborhood $V$ of $(x_0,\xi_0)$, where $c'$ and $V$ only depends on $c$, the growth of $a$ and the size of $U$.
\end{lem}
\begin{proof}
We shall write $T_h'u=T_h'(T_h^\ast T_hu)=(T_h'T_h^\ast)T_hu$.
\begin{equation*}
T_h'u(x,\xi)=\iiint e^{\frac{i}{h}(x-y)\xi-\frac{\mu}{2h}(x-y)^2-\frac{i}{h}(\tilde{x}-y)\tilde{\xi}-\frac{1}{2h}(\tilde{x}-y)^2}
a(x,y,\xi)T_hu(\tilde{x},\tilde{\xi})d\tilde{x}d\tilde{\xi}dy.
\end{equation*}
Therefore
\begin{equation}\label{eq:chfbi}
T_h'u(x,\xi)=\iint b(x,\xi,\tilde{x},\tilde{\xi})T_hu(\tilde{x},\tilde{\xi})d\tilde{x}d\tilde{\xi}
\end{equation}
where
\begin{equation*}
b(x,\xi,\tilde{x},\tilde{\xi})=e^{\frac{1}{\mu+1}\frac{i}{h}(x-\tilde{x})(\xi+\mu\tilde{\xi})
-\frac{\mu}{\mu+1}\frac{1}{2h}(x-\tilde{x})^2-\frac{1}{\mu+1}\frac{1}{2h}(\xi-\tilde{\xi})^2}
\int e^{-\frac{\mu+1}{2h}(y-\frac{\mu x+\tilde{x}}{\mu+1}+i\frac{\xi-\tilde{\xi}}{\mu+1})^2}a(x,y,\xi)dy
\end{equation*}
Now we change the contour to $y\mapsto y+\frac{\mu x+\tilde{x}}{\mu+1}-i\frac{\xi-\tilde{\xi}}{\mu+1}$,
\begin{equation*}
b(x,\xi,\tilde{x},\tilde{\xi})=e^{\frac{1}{\mu+1}\frac{i}{h}(x-\tilde{x})(\xi+\mu\tilde{\xi})
-\frac{\mu}{\mu+1}\frac{1}{2h}(x-\tilde{x})^2-\frac{1}{\mu+1}\frac{1}{2h}(\xi-\tilde{\xi})^2}
\int e^{-\frac{\mu+1}{2h}y^2}a(x,y+\frac{\mu x+\tilde{x}}{\mu+1}-i\frac{\xi-\tilde{\xi}}{\mu+1},\xi)dy
\end{equation*}
and use the assumption that $a$ is of tempered growth in $x,y,\xi$, we have
\begin{equation*}
|b(x,\xi,\tilde{x},\tilde{\xi})|\leqslant Ce^{-\frac{\delta}{h}(x-\tilde{x})^2-\frac{\delta}{h}(\xi-\tilde{\xi})^2}.
\end{equation*}
where $\delta,C$ depends only on the growth of $a$. Now the theorem follows easily from $\eqref{eq:chfbi}$ by separating the integral to two parts: $(\tilde{x},\tilde{\xi})$ close to $(x,\xi)$ and $(\tilde{x},\tilde{\xi})$ far away from $(x,\xi)$.
\end{proof}

\subsection{Equivalence between Cauchy estimates and decay of the FBI transform}
\label{sec:eqi}
In this section, we prove the equivalence between the semiclassical Cauchy estimate and the uniform exponential decay for the FBI transform when $|\xi|$ is large. Comparing to \cite{M}, we use different parameters for the FBI transform and the function $u$ itself, so we can capture both the microlocal and semiclassical properties of $u$. The idea of the proof is similar to the proof of the fact the projection of analytic wavefront set is the analytic singular support, see \cite{Sj}.

\begin{prop}\label{pr:equ}
Let $u=u(h),0<h\leqslant h_0$ be a family of function on a neighborhood $X$ of $x_0\in\mathbb{R}^n$ such that
$\|u\|_{L^\infty(X)}=O(h^{-N})$. Then the following are equivalent:\\
(i) There exists an open neighborhood $V\subset\subset X$ of $x_0$ and constants $C_0,C_1,\delta>0$ such that for every
$0<\tilde{h}\leqslant h\leqslant h_0$, $x\in V$ and $|\xi|\geqslant C_0$,
\begin{equation}\label{eq:fbies}
|T_{\tilde{h}}u(x,\xi,h)|\leqslant C_1e^{-\frac{\delta}{\tilde{h}}}\|u\|_{L^\infty}.
\end{equation}
(ii) There exists a complex neighborhood $W\subset\subset X+i\mathbb{R}^n$ of $x_0$ and a constant $C,C_2>0$ such that $u(h)$ can be extended holomorphic to $W$ and
\begin{equation}\label{eq:comes}
\sup_W|u(h)|\leqslant C_2e^{\frac{C}{h}}\|u\|_{L^\infty}.
\end{equation}
(iii) There exists an open neighborhood $U\subset\subset X$ of $x_0$, there exists $C_3>0$ such that for all $x\in U$,
\begin{equation}\label{eq:caues}
|(hD)^\alpha u(x)|\leqslant C_3^{|\alpha|}(1+h|\alpha|)^{|\alpha|}\|u\|_{L^\infty}.
\end{equation}
\end{prop}

\begin{proof}
First we notice that all of the statements are local, so we can extend $u$ to functions on $\mathbb{R}^n$, say by setting $u=0$ outside $X$, or better, to a family of functions in $C_0^\infty$ since each condition implies that $u$ is smooth (in fact, analytic) near $x_0$. Also if (ii) is true, then by Hadamard's three line theorem, there exists new constants $C,C_2>0$ such that
\begin{equation}\label{eq:comes2}
|u(z)|\leqslant C_2h^{-N}e^{\frac{C|\Im z|}{h}}.
\end{equation}

To prove that (ii) and (iii) are equivalent, we need the following elementary inequalities:
\begin{equation}\label{eq:e}
\forall t,s>0, (1+\frac{s}{t})^t\leqslant e^s\leqslant(1+\frac{s}{t})^{t+s}.
\end{equation}
\begin{equation}\label{eq:f}
\forall\alpha\in\mathbb{N}^n, (ne)^{-|\alpha|}|\alpha|^{|\alpha|}\leqslant\alpha!\leqslant|\alpha|^{|\alpha|}.
\end{equation}\\
Proof of (ii)$\Rightarrow$(iii): We can find a real neighborhood $U\subset\subset X$ of $x_0$ and a constant $r_0>0$ such that for all $x\in U$, the polydisc $D(x,r_0)\subset\subset W$, then by Cauchy's inequality (see \cite{H2} Theorem 2.2.7.) on $D(x,r)$ we have for a new constant $C>0$,
\begin{equation}\label{eq:cau}
|(hD)^\alpha u(x)|\leqslant Ch^{-N}e^{\frac{Cr}{h}}h^{|\alpha|}\alpha!r^{-|\alpha|}, 0<r\leqslant r_0, x\in U.
\end{equation}

Case 1: If $r_0\geqslant\frac{h|\alpha|}{C}$, then we take $r=\frac{h|\alpha|}{C}$ in \eqref{eq:cau} and get
\begin{equation*}
|(hD)^\alpha u(x)|\leqslant Ch^{-N}e^{|\alpha|}\alpha!C^{|\alpha|}|\alpha|^{-|\alpha|}
\end{equation*}
Now by \eqref{eq:f}, we have for a new constant $C>0$,
\begin{equation*}
|(hD)^\alpha u(x)|\leqslant C^{|\alpha|}h^{-N}.
\end{equation*}
This implies \eqref{eq:caues}.\\

Case 2: If $r_0<\frac{h|\alpha|}{C}$, then we take $r=r_0$ in \eqref{eq:cau} and get
\begin{equation*}
|(hD)^\alpha u(x)|\leqslant Ch^{-N}e^{\frac{Cr_0}{h}}h^{|\alpha|}\alpha!r_0^{-|\alpha|}
\end{equation*}
We use \eqref{eq:e} for $s=Cr_0,t=h|\alpha|$ and \eqref{eq:f}. Then
\begin{equation*}
\begin{split}
|(hD)^\alpha u(x)|&\leqslant Ch^{-N}(1+\frac{Cr_0}{h|\alpha|})^{\frac{Cr_0}{h}+|\alpha|}h^{|\alpha|}|\alpha|^{|\alpha|}r_0^{-|\alpha|}\\
&=Ch^{-N}(1+\frac{Cr_0}{h|\alpha|})^{\frac{Cr_0}{h}}(Cr_0+h|\alpha|)^{|\alpha|}r_0^{-|\alpha|}
\end{split}
\end{equation*}
which also implies \eqref{eq:caues} by our assumption $\frac{Cr_0}{h}<|\alpha|$.\\\\

Proof of (iii)$\Rightarrow$(ii): For $\delta>0$ small enough, $B(x_0,\delta)=\{|x-x_0|<\delta\}\subset U$, then for $x\in B(x_0,\delta)$, by Taylor's theorem,
\begin{equation*}
u(x)=\sum_{0\leqslant|\alpha|\leqslant k-1}\frac{\partial^\alpha u(x_0)}{\alpha!}(x-x_0)^\alpha+R_k
\end{equation*}
where
\begin{equation*}
R_k=\sum_{|\alpha|=k}\frac{1}{\alpha!}(x-x_0)^{\alpha}\int_0^1k(1-t)^{k-1}\partial^\alpha u(x_0+t(x-x_0))dt
\end{equation*}
Therefore by \eqref{eq:caues},
\begin{equation*}
|R_k|\leqslant\sum_{|\alpha|=k}\frac{1}{\alpha!}\delta^kC_3^k(1+hk)^kh^{-k}
\end{equation*}
We use \eqref{eq:e} and \eqref{eq:f} again to get
\begin{equation*}
|R_k|\leqslant(k+1)^n(C_3ne\delta)^k(1+\frac{1}{hk})^k\leqslant e^{\frac{1}{h}}(k+1)^n(C_3ne\delta)^k
\end{equation*}
Therefore as long as $\delta<(C_3ne)^{-1}$, $R_k\to0$ as $k\to\infty$, so $u$ is analytic on $B(x_0,\delta)$. Now we can extend $u$ holomorphically to $W=\{z\in\mathbb{C}^n:|z-x_0|<\delta\}$ by
\begin{equation}
u(z)=\sum_\alpha\frac{\partial^\alpha u(x_0)}{\alpha!}(z-x_0)^\alpha.
\end{equation}
Since
\begin{equation*}
\left|\frac{\partial^\alpha u(x_0)}{\alpha!}(z-x_0)^\alpha\right|\leqslant\frac{C_3^{|\alpha|}(1+h|\alpha|)^{|\alpha|}}{h^{|\alpha|}\alpha!}\delta^{|\alpha|}.
\end{equation*}
We apply \eqref{eq:e} for $s=1,t=h|\alpha|$ and \eqref{eq:f} to get
\begin{equation*}
\left|\frac{\partial^\alpha u(x_0)}{\alpha!}(z-x_0)^\alpha\right|\leqslant(C_3ne\delta)^{|\alpha|}
(1+\frac{1}{h|\alpha|})^{|\alpha|}\leqslant(C_3ne\delta)^{|\alpha|}e^{\frac{1}{h}}.
\end{equation*}
Thus
\begin{equation*}
|u(z)|\leqslant\sum_\alpha\left|\frac{\partial^\alpha u(x_0)}{\alpha!}(z-x_0)^\alpha\right|
\leqslant e^{\frac{1}{h}}\sum_\alpha(C_3ne\delta)^{|\alpha|}.
\end{equation*}
which gives \eqref{eq:comes} since $\delta<(C_3ne)^{-1}$.\\\\

Now we turn to the proof of (i)$\Leftrightarrow$(ii). We use the same type of deformation of the integral contour as in the proof that the projection of analytic wavefront set is the analytic singular support (see \cite{Sj}).\\
Proof of (ii)$\Rightarrow$(i): We have \eqref{eq:comes2} for $z$ in a neighborhood of $x_0$, say $\{z=y+it:|y-x_0|<2r,|t|<r\}$. For $|x-x_0|<r$, in the formula of FBI transform \eqref{eq:fbi},
\begin{equation*}
T_{\tilde{h}}u(x,\xi,h)=2^{-\frac{n}{2}}(\pi\tilde{h})^{-\frac{3n}{4}}\int e^{\frac{i}{\tilde{h}}(x-y)\xi-\frac{1}{2\tilde{h}}(x-y)^2}u(y)dy,
\end{equation*}
we deform the contour to
\begin{equation}
\Gamma_x:y\mapsto z=y+i\epsilon\chi(y)\frac{\xi}{|\xi|}.
\end{equation}
where $\chi\in C_0^\infty(\mathbb{R}^n), 0\leqslant\chi\leqslant1,\chi=1 $ on $|y-x|<\frac{r}{2}$, $\supp\chi\subset\{|y-x|<r\}$ and $\epsilon\in(0,r)$. Then along $\Gamma_x$,
\begin{equation*}
\left|e^{\frac{i}{\tilde{h}}(x-z)\xi-\frac{1}{2\tilde{h}}(x-z)^2}u(z)\right|\leqslant Ch^{-N}e^{\frac{C}{h}\epsilon\chi(y)-\frac{1}{\tilde{h}}\epsilon\chi(y)|\xi|+\frac{1}{2\tilde{h}}\epsilon^2\chi(y)^2-\frac{|x-y|^2}{2\tilde{h}}}.
\end{equation*}
Since
\begin{equation*}
\frac{C}{h}\epsilon\chi(y)-\frac{1}{\tilde{h}}\epsilon\chi(y)|\xi|+\frac{1}{2\tilde{h}}\epsilon^2\chi(y)^2-\frac{|x-y|^2}{2\tilde{h}}
\leqslant\frac{1}{\tilde{h}}\epsilon\chi(y)[C+\frac{\epsilon}{2}-|\xi|]-\frac{|x-y|^2}{2\tilde{h}}
\end{equation*}
we have if $|\xi|>C_0=C+\frac{\epsilon}{2}+\frac{\delta}{\epsilon}$,
\begin{equation*}
\left|e^{\frac{i}{\tilde{h}}(x-z)\xi-\frac{1}{2\tilde{h}}(x-z)^2}u(z)\right|\leqslant
\left\{\begin{array}{ll}
Ch^{-N}e^{-\frac{\delta}{\tilde{h}}}, & \text{ when } |y-x|<\frac{r}{2}\\
Ch^{-N}e^{-\frac{r^2}{8\tilde{h}}}, & \text{ when } |y-x|\geqslant\frac{r}{2}
\end{array}\right.
\end{equation*}
which shows \eqref{eq:fbies}.\\

Proof of (i)$\Rightarrow$(ii): We have
\begin{equation*}
\delta(x)=(2\pi h)^{-n}\int e^{\frac{i}{h}x\xi}d\xi
\end{equation*}
in the sense of oscillatory integral. Following Lebeau, we deform to the complex contour
\begin{equation}
\tilde{\Gamma}_x: \xi\mapsto\zeta=\xi+\frac{i}{2}|\xi|x
\end{equation}
Along $\tilde{\Gamma}_x$,
\begin{equation*}
d\zeta=a(x,\xi)d\xi, a(x,\xi)=1+\frac{i}{2}\sum_1^n\frac{x_j\xi_j}{|\xi|}.
\end{equation*}
Therefore in the sense of oscillatory integral,
\begin{equation}
\delta(x)=(2\pi h)^{-n}\int e^{\frac{i}{h}x\xi-\frac{1}{2h}|\xi|x^2}a(x,\xi)d\xi.
\end{equation}
Now we can write $u$ in the form of
\begin{equation}
u(x)=(2\pi h)^{-n}\iint e^{\frac{i}{h}(x-y)\xi-\frac{1}{2h}|\xi|(x-y)^2}a(x-y,\xi)u(y,h)dyd\xi.
\end{equation}
Let
\begin{equation*}
I(x,\xi,h)=\int e^{\frac{i}{h}(x-y)\xi-\frac{1}{2h}|\xi|(x-y)^2}a(x-y,\xi)u(y,h)dy,
\end{equation*}
We claim that
\begin{equation}\label{eq:expdecay}
I(x+it,\xi,h)=O(e^{\frac{C}{h}-\frac{c|\xi|}{h}}), C,c>0
\end{equation}
uniformly for $x+it$ in a complex neighborhood of $x_0$ and $0<h\leqslant h_0$. In fact,
\begin{equation}\label{eq:deint}
I(x+it,\xi,h)=e^{-\frac{t\xi}{h}+\frac{t^2|\xi|}{2h}}\int e^{\frac{i}{h}(x-y)(\xi-|\xi|t)-\frac{1}{2h}|\xi|(x-y)^2}a(x+it-y,\xi)u(y)dy
\end{equation}
It is easy to see when $|\xi|<C'$, $I(x+it,\xi,h)=O(e^{\frac{C}{h}})$, so we have \eqref{eq:expdecay}. Now we assume $|\xi|>C'$ where $C'$ we shall choose to be large later, then since $|t|<\epsilon$,
\begin{equation*}
(1-\epsilon)|\xi|<|\xi-|\xi|t|<(1+\epsilon)|\xi|.
\end{equation*}
Let $\tilde{h}=\frac{\mu h}{|\xi|}, \mu$ large and fixed later, then we can rewrite \eqref{eq:deint} as
\begin{equation*}
I(x+it,\xi,h)=e^{-\frac{t\xi}{h}+\frac{t^2|\xi|}{2h}}\int e^{\frac{i}{\tilde{h}}(x-y)\cdot(\frac{\tilde{h}}{h}(\xi-|\xi|t))-\frac{\mu}{2\tilde{h}}(x-y)^2}a(x+it-y,\xi)u(y)dy
\end{equation*}
Now we choose $C'>\mu>C_0(1-\epsilon)^{-1}$, then $\tilde{h}\leqslant h$ and
\begin{equation*}
\left|\frac{\tilde{h}}{h}(\xi-|\xi|t)\right|=\mu\frac{|\xi-|\xi|t|}{|\xi|}\geqslant(1-\delta)\mu>C_0
\end{equation*}
By Lemma \ref{pr:chfbi} (we notice that $y\mapsto a(x+it-y,\xi)$ has uniform tempered growth when $|t|$ is small) and \eqref{eq:fbies}, we have uniform exponential decay for the integral in \eqref{eq:deint} when $x+it$ is in a small complex neighborhood of $x_0$, and $|\xi|>C'$,
\begin{equation*}
I(x+it,\xi,h)\leqslant Ce^{-\frac{t\xi}{h}+\frac{t^2|\xi|}{2h}}e^{-\frac{\delta'}{\tilde{h}}}\leqslant Ce^{-\frac{c|\xi|}{h}}.
\end{equation*}
if we assume $|t|<\epsilon$ is small enough. This finishes the proof of \eqref{eq:expdecay}.

Now we can extend $u$ holomorphically to a complex neighborhood of $x_0$ simply by
\begin{equation}
u(z)=(2\pi h)^{-n}\int I(z,\xi,h)d\xi.
\end{equation}
since $I(z,\xi,h)$ is holomorphic and the integral is uniformly convergent. Furthermore,
\begin{equation}
|u(z)|\leqslant C(2\pi h)^{-n}\int e^{\frac{C}{h}-\frac{c|\xi|}{h}}d\xi\leqslant C_2e^{\frac{C}{h}}.
\end{equation}
which gives \eqref{eq:comes}.
\end{proof}

\begin{rem}
Since $e^{\frac{\xi^2}{2\tilde{h}}}T_{\tilde{h}}u(x,\xi;h)$ is holomorphic, we can replace condition (i) by the exponential decay of local $L^2$-norm of $T_{\tilde{h}}u(h)$.
\end{rem}

\subsection{Agmon estimates for the FBI transform}
\label{sec:ag}

We shall follow the approach in \cite{M}. First we recall the following theorem of microlocal exponential estimate from \cite[Corollary 3.5.3, $f=1$]{M}:

\begin{thm}
Suppose $p\in S_{2n}(1)$ can be extended holomorphically to
\begin{equation*}
\Sigma(a)=\{(x,\xi)\in\mathbb{C}^{2n}:|\Im x|<a,|\Im\xi|<a\}
\end{equation*}
such that
\begin{equation*}
\forall\alpha\in\mathbb{N}^{2n},\partial^\alpha p=O(1), \text{ uniformly in } \Sigma(a).
\end{equation*}
Assume also that the real-valued function $\psi\in S_{2n}(1)$ satisfies
\begin{equation*}
\sup_{\mathbb{R}^{2n}}|\nabla_x\psi|<a, \sup_{\mathbb{R}^{2n}}|\nabla_\xi\psi|<a.
\end{equation*}
Then
\begin{equation*}
\|e^{\psi/h}T_hP(x,hD)u\|^2=\|p(x-2\partial_z\psi,\xi+2i\partial_z\psi)e^{\psi/h}T_hu\|^2+O(h)\|e^{\psi/h}T_hu\|^2
\end{equation*}
uniformly for $u\in L^2(\mathbb{R}^n)$, $h>0$ small enough. Here $\partial_z=\frac{1}{2}(\partial_x+i\partial_\xi)$ is the holomorphic derivative with respect to $z=x-i\xi$.
\end{thm}

\begin{rem}
From the argument in \cite{M}, we can also see that this estimate only depends on the seminorms of $p$ and $\psi$ in $S_{2n}(1)$. In other words, if $p$ and $\psi$ varies in a way such that every $\sup\limits_{\Sigma(a,b)}|\partial^\alpha p|$ and $\sup\limits_{\mathbb{R}^{2n}}|\partial^\alpha\psi|$ is uniformly bounded, then the estimate is uniform in $p$ and $\psi$. Furthermore, we only need that $p$ can be extended holomorphically to the set $\{(y,\eta)\in\mathbb{C}^{2n}:\exists (x,\xi)\in\supp\psi, |y-x|<\sup|\nabla\psi|, |\eta-\xi|<\sup|\nabla\psi|\}$. Also here $P(x,hD)$ can be replaced by any quantization as in \cite{M},\cite{Z}.
\end{rem}

Now we consider a semiclassical differential operator $P=P(x,hD_x)$ of order $m$ with analytic coefficients, defined in a neighborhood $X$ of $x_0$. We assume the symbol $$p(x,\xi;h)=\sum_{|\alpha|\leqslant m}a_\alpha(x;h)\xi^\alpha$$
can be extended holomorphically to a fixed complex neighborhood
$$\Sigma_\delta=\{(x,\xi)\in\mathbb{C}^n\times\mathbb{C}^n:|\Re x-x_0|<\delta, |\Im x|<\delta, |\Im\xi|<\delta\}$$
and also that $P$ is classically elliptic in $\Sigma_\delta$, in the sense that the principal symbol
$$p_0(x,\xi)=\sum_{|\alpha|=m}a_\alpha\xi^\alpha$$
satisfies
$$|p_0(x,\xi)|\geqslant\frac{1}{C_0}\langle\Re\xi\rangle^m, \text{ for } (x,\xi)\in\Sigma_\delta, |\Re\xi|>C$$

\begin{thm}
Let $P$ be as above and assume $\{u(h)\}_{0<h\leqslant h_0}$ is a family of functions defined in $X$ such that
$$P(x,hD_x)u(h)=0 \text{ in } X$$
and
$$\|u(h)\|_{L^2(X)}\leqslant Ch^{-N}$$
Then there exists an open neighborhood $U\subset\subset X$ of $x_0$, such that for all $x\in U$,
\begin{equation}\label{eq:cauchy}
|(hD)^\alpha u(x)|\leqslant C^{|\alpha|}(1+h|\alpha|)^{|\alpha|}\|u(h)\|_{L^\infty}.
\end{equation}
\end{thm}

\begin{rem}
Also by the standard semiclassical elliptic estimates, (e.g. \cite[Lemma 7.10]{Z}), we know $\|u\|_{L^\infty}\leqslant Ch^{-n/2}\|u\|_{L^2}$. So we also have $\|u(h)\|_{L^\infty}\leqslant Ch^{-M}$.
\end{rem}

\begin{proof}
First for $0<\tilde{h}\leqslant h$, we write
$$\tilde{P}(x,\tilde{h}D_x)=(\tilde{h}/h)^m\chi_1(x)P(x,hD_x)$$
so
$$\tilde{p}(x,\xi)=\sum_{|\alpha|\leqslant m}(\tilde{h}/h)^{m-|\alpha|}\chi_1(x)a_\alpha(x;h)\xi^\alpha$$
where $\chi_1\in C_0^\infty(\mathbb{R}^n)$ is a cut-off function satisfying $\chi_1(x)=1$ for $|x-x_0|<\frac{\delta}{4}$; 0 for $|x-x_0|>\frac{\delta}{2}$. Therefore $\tilde{p}(x,\xi)$ can still be extended holomorphically to $\Sigma_{\frac{\delta}{4}}$ and classically elliptic in $\Sigma_{\frac{\delta}{4}}$. Now let $Q=\langle\tilde{h}D_x\rangle^{-m}\circ\tilde{P}(x,\tilde{h}D_x)$, then we can write $Q(x,\tilde{h}D_x)=\Op_{\tilde{h}}^1(q)$ where
$$q(y,\xi)=\langle\xi\rangle^{-m}\sum_{|\alpha|\leqslant m}\sum_{\beta\leqslant\alpha}(\tilde{h}D_y)^\beta(\chi_1(y)a_\alpha(y;h))(-\xi)^{\alpha-\beta}.$$
In fact,
$$Qv=(2\pi\tilde{h})^{-n}\iint e^{\frac{i}{h}(x-y)\xi}\langle\xi\rangle^{-m}(\tilde{P}v)(y)dyd\xi
=(2\pi\tilde{h})^{-n}\iint e^{\frac{i}{h}(x-y)\xi}q(y,\xi)v(y)dyd\xi.$$
Therefore $q$ can also be extended holomorphically to $\Sigma_{\frac{\delta}{4}}$ and
$$|\partial^\alpha q|=O_\alpha(1), \text{ in } \Sigma_{\frac{\delta}{4}}.$$
Furthermore $q$ is elliptic in $\Sigma_{\frac{\delta}{3}}$ and thus for $\tilde{h}$ small,
$$|q(x,\xi)|\geqslant\frac{1}{C}, \text{ for } (x,\xi)\in\Sigma_{\frac{\delta}{4}}, |\xi|>C_0.$$
Now let $v(x)=\chi_2(x)u(x)$, where $\chi_2(x)\in C_0^\infty(\mathbb{R}^n)$ is a cut-off function satisfying $0\leqslant\chi_2\leqslant1$, $\chi_2(x)=1$ for $|x-x_0|<\frac{\delta}{2}$; 0 for $|x-x_0|>\frac{3\delta}{4}$. Therefore $v\in L^2(\mathbb{R}^n)$ and
$$Qv=0$$
$$\|v\|_{L^2(\mathbb{R}^n)}\leqslant\|u\|_{L^2(X)}\leqslant Ch^{-N}.$$
Now we choose $\psi=\psi(x,\xi)\in S_{2n}(1)$ such that $\supp\psi\subset U_1=\{|x-x_0|<\frac{\delta}{8},|\xi|>2C_0\}$ and $\psi=c>0$ on $U_2=\{|x-x_0|<\frac{\delta}{16},|\xi|>3C_0\}$ and $\sup_{\mathbb{R}^{2n}}|\nabla_{(x,\xi)}\psi|<\frac{\delta}{16}$. Then if $(x,\xi)\in\supp\psi$, then for any $(y,\eta)\in\mathbb{C}^{2n}$ satisfying $|y-x|<\frac{\delta}{16},|\eta-\xi|<\frac{\delta}{16}$, we have $(y,\eta)\in\Sigma_{\frac{\delta}{4}}$. This allows us to apply the microlocal exponential estimate for $\psi$ and $q$:
\begin{equation*}
0=\|e^{\psi/\tilde{h}}T_{\tilde{h}}Qv\|^2=\|q(x-2\partial_z\psi,\xi+2i\partial_z\psi)e^{\psi/\tilde{h}}T_{\tilde{h}}v\|^2
+O(\tilde{h})\|e^{\psi/\tilde{h}}T_{\tilde{h}}v\|^2
\end{equation*}
Therefore
\begin{equation*}
\|q(x-2\partial_z\psi,\xi+2i\partial_z\psi)e^{\psi/\tilde{h}}T_{\tilde{h}}v\|_{L^2(U_1)}^2
=O(\tilde{h})\|e^{\psi/{\tilde{h}}}T_{\tilde{h}}v\|^2
\end{equation*}
For $(x,\xi)\in U_1$, $(x-2\partial_z\psi,\xi+2i\partial_z\psi)\in\Sigma_{\frac{\delta}{4}}$ and $|\Re(\xi+2i\partial_z\psi)|\geqslant2C_0-\frac{\delta}{16}>C_0$, so
$$|q(x-2\partial_z\psi,\xi+2i\partial_z\psi)|\geqslant\frac{1}{C}.$$
Therefore
\begin{equation*}
\|e^{\psi/\tilde{h}}T_{\tilde{h}}v\|_{L^2(U_1)}^2=O(\tilde{h})\|e^{\psi/{\tilde{h}}}T_{\tilde{h}}v\|^2
\end{equation*}
When $\tilde{h}$ is small, we have
\begin{equation*}
\|e^{\psi/\tilde{h}}T_{\tilde{h}}v\|_{L^2(U_1)}^2=O(\tilde{h})\|e^{\psi/{\tilde{h}}}T_{\tilde{h}}v\|^2_{L^2(\mathbb{R}^{2n}\setminus U_1)}
\end{equation*}
Since $\psi=0$ outside $U_1$, we have
\begin{equation*}
\|e^{\psi/\tilde{h}}T_{\tilde{h}}v\|_{L^2(U_1)}^2\leqslant C\tilde{h}\|T_{\tilde{h}}v\|^2_{L^2(\mathbb{R}^{2n}\setminus U_1)}\leqslant C\tilde{h}\|T_{\tilde{h}}v\|^2=C\tilde{h}\|v\|^2\leqslant C\tilde{h}\|u\|^2.
\end{equation*}
Since $\psi=c>0$ on $U_2\subset U_1$,
$$\|T_{\tilde{h}}v\|_{L^2(U_2)}\leqslant Ce^{-\frac{c}{\tilde{h}}}\|u\|^2.$$
Now by Proposition \ref{pr:equ} and the remark after it, we can conclude the proof of \eqref{eq:cauchy}.
\end{proof}

\begin{rem}
The same argument can also be applied to elliptic pseudodifferential operators on $\mathbb{R}^n$ with symbol
\begin{equation}
p(x,\xi;h)\sim\sum_{j=0}^\infty a_m(x,\xi/|\xi|;h)|\xi|^{m-j}
\end{equation}
which can be extended holomorphically to $\Sigma(a)$ for some $a>0$ and is classically elliptic in $\Sigma(a)$. In this case, we do not need any cut-off function and the weight function $\psi$ can be chosen to only depend on $\xi$. Then the solutions of $P(x,hD_x)u(h)=0$ in $\mathbb{R}^n$ also satisfies the semiclassical Cauchy estimates \eqref{eq:cauchy}.
\end{rem}

\section{Doubling property}
\label{sec:dp}

In this section, we use a Carleman-type estimate to prove the so-called doubling property of solutions of semiclassical Schrodinger equations on a compact Riemannian manifold. We do not assume the analyticity of either the manifold or the potential. See \cite{B-C} for a general setting where the potential is only assumed to be $C^1$. From now on, for simplicity, we shall use $\|u\|_U$ to represent the $L^2$-norm of the function $u$ in the set $U$.

\begin{thm}\label{pr:dp}
Suppose $(M,g)$ is a compact Riemannian manifold, $V:M\to\mathbb{R}$ is a smooth function. Let $P(h)=-h^2\Delta_g+V(x)$ and $P(h)u(h)=E(h)u(h)$ where $E(h)\to E_0$ as $h\to0$, then\\
(i) (Tunneling) For every $r>0$, there exists $c=c(r)>0$ depending on $r$ (and $(M,g),V$) such that
\begin{equation}\label{eq:tun}
\|u\|_{L^2(B(p,r))}\geqslant e^{-c(r)/h}\|u\|_{L^2(M)}
\end{equation}
for every $p\in M$, $0<h<h_0$.\\
(ii) (Doubling Property) There exists $c_0>0$ depending only on $(M,g)$ and $V$ and such that for every $c_1>0$,
\begin{equation}\label{eq:dp}
\|u\|_{L^2(B(p,r))}\geqslant e^{-c_0/h}\|u\|_{L^2(B(p,2r))}
\end{equation}
uniformly for $p\in M, r>c_1h$ and $0<h<h_1$.\\
\end{thm}

\begin{rem}
We can remove the condition $r>c_1h$ in part (ii) by carefully constructing a weight involving logarithmic terms near the origin in Carleman estimates. For the details, see \cite{B-C}. For our purpose, the weak version above will be sufficient.
From now on, in this section, every constant will depend on $(M,g)$ and $V$, but we shall not write it out explicitly.
\end{rem}

\subsection{Carleman estimates}
\label{sec:ce}

We start by writing the equation in the local coordinates. Let $r_1$ be the injective radius of $M$, then for any $p\in M$, we write $P(h)u(h)=0$ on $B(p,r_1)$ in the normal geodesic coordinates centered at $p$ still as
\begin{equation}\label{eq:loc}
[-h^2\Delta_g+V(x)-E]u=0, x\in B(0,r_1).
\end{equation}

Let $p=|\xi|^2_g+V(x)-E$ be the symbol of $P=-h^2\Delta_g+V(x)-E$. We wish to conjugate $P$ by a weight $e^{\varphi/h}$ to get an operator $P_\varphi=e^{\varphi/h}Pe^{-\varphi/h}$ whose symbol $p_\varphi=|\xi+i\partial\varphi(x)|^2_{g(x)}+V(x)-E$ satisfies H\"{o}rmander's hypoelliptic condition:
\begin{equation}
\text{ if } p_\varphi=0, \text{ then } \frac{i}{2}\{p_\varphi,\bar{p}_\varphi\}=\{\Re p_\varphi,\Im p_\varphi\}>0.
\end{equation}
on $B(0,R)\setminus B(0,r)$. Here $\{\cdot,\cdot\}$ denotes the Poisson bracket.

Since
$$\Re p_\varphi=|\xi|^2_g-|\partial\varphi|^2_g+V-E, \Im p_\varphi=2\langle\xi,\partial\varphi\rangle_g.$$
We have
\begin{equation*}
\begin{split}
\{\Re p_\varphi,\Im p_\varphi\}=&\langle \partial_\xi(\Re p_\varphi),\partial_x(\Im p_\varphi)\rangle-\langle\partial_x(\Re p_\varphi),\partial_\xi(\Im p_\varphi)\rangle\\
=&4\langle g\xi,\partial^2\varphi g\xi\rangle+4\langle g\partial\varphi,\partial^2\varphi g\partial\varphi\rangle-2\langle g\partial\varphi,\partial V\rangle\\
&+4\langle g\xi,\langle\xi,\partial\varphi\rangle_{\partial g}\rangle-2\langle g\partial\varphi,|\xi|^2_{\partial g}-|\partial\varphi|^2_{\partial g}\rangle
\end{split}
\end{equation*}
We set $\varphi=\tau e^{\mu\psi}$, where $\tau,\mu\geqslant1$ is a large constant to be chosen later. Then
$$\partial\varphi=\tau\mu e^{\mu\psi}\partial\psi, \partial^2\varphi=\tau e^{\mu\psi}(\mu^2\partial\psi\otimes\partial\psi+\mu\partial^2\psi).$$
Therefore
\begin{equation*}
\begin{split}
\frac{i}{2}\{p_\varphi,\bar{p}_\varphi\}=&\; 4\tau\mu^2e^{\mu\psi}\langle g\xi,\partial\psi\rangle^2+4\tau\mu e^{\mu\psi}\langle g\xi,\partial^2\psi g\xi\rangle\\
&+4\tau^3\mu^4e^{3\mu\psi}\langle\partial\psi,g\partial\psi\rangle^2+4\tau^3\mu^3e^{3\mu\psi}\langle g\partial\psi,\partial^2\psi g\partial\psi\rangle\\
&-2\tau\mu e^{\mu\psi}\langle g\partial\psi,\partial V\rangle+4\tau\mu e^{\mu\psi}\langle g\xi,\langle\xi,\partial\psi\rangle_{\partial g}\rangle\\
&-2\tau\mu e^{\mu\psi}\langle g\partial\psi,|\xi|^2_{\partial g}\rangle+2\tau^3\mu^3e^{3\mu\psi}\langle g\partial\psi,|\partial\psi|^2_{\partial g}\rangle
\end{split}
\end{equation*}

When $\Re p_\varphi=\Im p_\varphi=0$, we have
\begin{equation*}
|\xi|^2_g=|\partial_\varphi|_g^2+V(x)-E=\tau^2\mu^2e^{2\mu\psi}|\partial\psi|_g^2+V-E
\end{equation*}
\begin{equation*}
\langle\xi,\partial\varphi\rangle_g=2\tau\mu e^{\mu\psi}\langle g\xi,\partial\psi\rangle=0
\end{equation*}
thus
\begin{equation*}
\xi=\tau\mu e^{\mu\psi}|\partial\psi|_g\eta, \text{ where } C_1^{-1}\leqslant|\eta|\leqslant C_1.
\end{equation*}
We shall choose $\psi$ to be a radial and radially decreasing function which equals to $A-|x|$ on $B(0,R)\setminus B(0,r)$ so that
\begin{equation*}
\psi\geqslant 1, C_2^{-1}\leqslant|\partial\psi|_g\leqslant C_2 \text{ and } |\partial^2\psi|\leqslant C_2 \text{ on } B(0,R)\setminus B(0,r)
\end{equation*}
Hence
\begin{equation*}
\frac{i}{2}\{p_\varphi,\bar{p}_\varphi\}\geqslant4\tau^3\mu^4e^{3\mu\psi}[|\partial\psi|_g^4-C_3\mu^{-1}]\geqslant C_4\tau^{-1}e^{-\mu\psi}\langle\psi\rangle^4
\end{equation*}
where $C_3,C_4$ is a constant depending only on $\psi$ when $\mu$ and $\tau$ are large depending on $\psi$. Now we can prove the basic Carleman estimate:

\begin{lem}
For any $v\in C_0^\infty(B(0,R)\setminus B(0,r))$,
\begin{equation}\label{eq:bcar}
C_5\tau^{\frac{1}{2}}\|P_\varphi v\|\geqslant h^{\frac{1}{2}}\|v\|_{H_h^2}.
\end{equation}
where $C_5$ is a constant only depending on $\mu, \psi$.
\end{lem}
\begin{proof}
The proof is based on the standard commutator argument. First,
\begin{equation*}
\begin{split}
\|P_\varphi v\|^2=&\langle P_\varphi v,P_\varphi v\rangle=\langle P_\varphi^\ast P_\varphi v,v\rangle\\
=&\langle P_\varphi^\ast P_\varphi v,v\rangle+\langle[P_\varphi^\ast,P_\varphi]v,v\rangle=\|P_\varphi^\ast v\|^2+\langle[P_\varphi^\ast,P_\varphi]v,v\rangle.
\end{split}
\end{equation*}
For any $M>1$ and $h$ small enough, this implies
\begin{equation*}
\begin{split}
\|P_\varphi v\|^2&\geqslant Mh\|P_\varphi^\ast v\|^2+\langle[P_\varphi^\ast,P_\varphi]v,v\rangle\\
&=h\langle\Op_h(M|p_\varphi|^2+i\{p_\varphi,\bar{p}_\varphi\})u,u\rangle-O(h^2)\|u\|_{H_h^2}
\end{split}
\end{equation*}
From the construction above, we can find $M$ large enough so that
\begin{equation*}
M|p_\varphi|^2+i\{p_\varphi,\bar{p}_\varphi\}\geqslant C_6\tau^{-1}\langle\xi\rangle^4
\end{equation*}
where $C_6$ is a constant only depending on $\mu,\psi$. Now we can use the sharp G{\aa}rding's inequality to conclude the lemma.
\end{proof}

Now we prove a Carleman estimate on different shells.
\begin{prop}[Carleman estimate on shells]\label{pr:carsh}
Let $\varphi$ be as above, we have the following estimate for solutions $u$ to the equation \eqref{eq:loc}
\begin{equation}\label{eq:car}
\begin{split}
\|e^{\varphi/h}u\|_{B(0,\frac{R}{2})\setminus B(0,2r)}\leqslant &\; C_7\tau^{\frac{1}{2}}[R^{-2}e^{\varphi(\frac{R}{2})/h}\|u\|_{B(0,R)\setminus B(0,\frac{R}{2})}\\
&\;\;\;\;\;+r^{-2}e^{\varphi(2r)/h}\|u\|_{B(0,2r)\setminus B(0,r)}].
\end{split}
\end{equation}
where $0<8r<R<r_1$, $C_7>0$ only depends on $\mu,\psi$.
\end{prop}
\begin{proof}
We shall take $v=e^{\varphi/h}\chi u$ in \eqref{eq:bcar} where $\chi$ is a function supported in $C_0^\infty(B(0,R)\setminus B(0,r))$, such that $\chi\equiv1$ on $B(0,\frac{R}{2})\setminus B(0,2r)$, and
\begin{equation*}
|\nabla\chi|\leqslant Cr^{-1}, |\nabla^2\chi|\leqslant Cr^{-2}.
\end{equation*}
Then
\begin{equation*}
\|v\|=\|e^{\varphi/h}\chi u\|\geqslant \|e^{\varphi/h}u\|_{B(0,\frac{R}{2})\setminus B(0,2r)},
\end{equation*}
and
\begin{equation*}
\|P_\varphi v\|=\|e^{\varphi/h}P(\chi u)\|=\|e^{\varphi/h}[P,\chi]u\|
\end{equation*}
where $[P,\chi]=\langle a,hD\rangle+b$ is a first order differential operator supported on $\supp\nabla\chi\subset(B(0,R)\setminus B(0,\frac{R}{2}))\cup(B(0,2r)\setminus B(0,r))$ with coefficients $a=O(hr^{-1}), b=O(h^2r^{-2})$ on $B(0,2r)\setminus B(0,r)$ and with $r$ replaced by $R$ on $B(0,R)\setminus B(0,\frac{R}{2})$. Therefore by standard elliptic estimates (e.g. \cite[Chapter 7]{Z}), we have
\begin{equation*}
\begin{split}
\|P_\varphi v\|\leqslant&\; C_8[e^{\varphi(\frac{R}{2})/h}\|u\|_{H_h^1(B(0,R)\setminus B(0,\frac{R}{2}))}+e^{\varphi(2r)/h}\|u\|_{H_h^1(B(0,2r)\setminus B(0,r))}]\\
\leqslant&\; C_9[R^{-2}h^2e^{\varphi(\frac{R}{2})/h}\|u\|_{B(0,R)\setminus B(0,\frac{R}{2})}+r^{-2}h^2e^{\varphi(r)/h}\|u\|_{B(0,2r)\setminus B(0,r)}].
\end{split}
\end{equation*}
which finishes the proof.
\end{proof}

\subsection{Proof of Theorem \ref{pr:dp}}
\label{sec:pr}
We shall use \ref{pr:carsh} to prove Theorem \ref{pr:dp}. The tunneling estimate follows from the standard overlapping chains of
balls argument introduced by Donnelly and Fefferman \cite{D-F} while the doubling property is a corollary of the tunneling and the Carleman estimates on shells.
\begin{proof}[Proof of \ref{eq:tun}]
Without loss of generality, we can assume $\|u\|_M=1$ and we only need to prove \ref{eq:tun} for $r<\frac{1}{100}r_1$. We shall fix $\tau$ large in the expression of $\varphi$ and replace $r$ by $\frac{r}{4}$ and take $R=32r$ in Proposition \ref{pr:carsh}. Then we get
\begin{equation}\label{eq:car1}
e^{\varphi(2r)/h}\|u\|_{B(x,2r)\setminus B(x,r)}
\leqslant C_{10}[e^{\varphi(4r)/h}\|u\|_{B(x,8r)\setminus B(x,4r)}+e^{\varphi(\frac{r}{4})/h}\|u\|_{B(0,\frac{r}{2})\setminus B(0,\frac{r}{4})}]
\end{equation}
for any $x\in M$. It is obvious that there exists a point $q\in M$ such that
$$\|u\|_{B(q,r)}\geqslant cr^{-n}\geqslant e^{-A_0/h}.$$
For any $p\in M$, we can find a sequence
$x_0=q, x_1,\ldots, x_m=p$ such that $d(x_j,x_{j+1})<r$ and $m\leqslant C_{11}r^{-1}$. We shall prove by induction that there exists $A_j>0, h_j>0$ only depends on $r$ such that for $0<h<h_j$
\begin{equation}\label{eq:ind}
\|u\|_{B(x_j,r)}\geqslant e^{-A_j/h}.
\end{equation}
We already know this is true for $j=0$. Suppose this is true for $j$, then for $j+1$, since $B(x_j,r)\subset B(x_{j+1},2r)$, either
$$\|u\|_{B(x_{j+1},r)}\geqslant e^{-(A_j+1)/h}$$
or
$$\|u\|_{B(x_{j+1},2r)\setminus B(x_{j+1},r)}\geqslant e^{-(A_j+1)/h}.$$
For the first case, there is nothing to prove, for the later, we let $x=x_{j+1}$ in \eqref{eq:car1} to get
\begin{equation}
\|u\|_{B(0,\frac{r}{2})\setminus B(0,\frac{r}{4})}\geqslant e^{-\varphi(\frac{r}{4})/h}[C_{12}^{-1}h^{-\frac{3}{2}}e^{\varphi(2r)/h}e^{-(A_j+1)/h}-e^{\varphi(4r)/h}].
\end{equation}
We only need to choose $\mu$ and $A$ large enough in the expression of $\varphi$ so that $\varphi(2r)-A_j-1>\varphi(4r)$ to get \eqref{eq:ind}. Now since $m$ is bounded, we get the desired tunneling estimates \ref{eq:tun}.
\end{proof}
\begin{proof}[Proof of \eqref{eq:dp}]
Again, we only need to prove \eqref{eq:dp} for $r\in(c_1h,\frac{r_1}{100})$. Now we shall fix $R=\frac{r_1}{2}$ in Proposition \ref{pr:carsh} and replace $r$ by $\frac{r}{2}$, then for any $p\in M$,
\begin{equation}\label{eq:car2}
\begin{split}
&e^{\varphi(\frac{R}{4})/h}\|u\|_{B(p,\frac{R}{4})\setminus B(p,\frac{R}{8})}+e^{\varphi(2r)/h}\|u\|_{B(p,2r)\setminus B(p,r)}\\
\leqslant&\;C_{13}\tau^{\frac{1}{2}}[e^{\varphi(\frac{R}{2})/h}\|u\|_{B(p,R)\setminus B(p,\frac{R}{2})}+h^{-2}e^{\varphi(\frac{r}{2})/h}\|u\|_{B(p,r)\setminus B(p,\frac{r}{2})}].
\end{split}
\end{equation}
By the tunneling estimates \eqref{eq:tun} and the fact that there exists a ball of radius $\frac{R}{8}$ inside $B(p,\frac{R}{4})\setminus B(p,\frac{R}{8})$,
\begin{equation*}
\|u\|_{B(p,\frac{R}{4})\setminus B(p,\frac{R}{8})}\geqslant e^{-C_{14}/h}\|u\|\geqslant e^{-C_{14}/h}\|u\|_{B(p,R)\setminus B(p,\frac{R}{2})}
\end{equation*}
By choosing $\mu$ and $A$ large enough only depending on $R$ and $\tau=C_{15}r^{-1}\leqslant C_{16}h^{-1}$, we can make $\varphi(\frac{R}{4})-C_{14}\geqslant\varphi(\frac{R}{2})$. Therefore for $h<h_0$,
\begin{equation*}
e^{\varphi(\frac{R}{4})/h}\|u\|_{B(p,\frac{R}{4})\setminus B(p,\frac{R}{8})}\geqslant C_{13}\tau^{\frac{1}{2}}e^{\varphi(\frac{R}{2})/h}\|u\|_{B(p,R)\setminus B(p,\frac{R}{2})}.
\end{equation*}
Therefore from \label{eq:car2} we see that
\begin{equation*}
C_{13}\tau^{\frac{1}{2}}h^{-2}e^{\varphi(\frac{r}{2})/h}\|u\|_{B(p,r)\setminus B(p,\frac{r}{2})}\geqslant e^{\varphi(2r)/h}\|u\|_{B(p,2r)\setminus B(p,r)}.
\end{equation*}
Since
$$\varphi(\frac{r}{2})-\varphi(2r)=\tau e^{\mu(A-2r)}(e^{3\mu r/2}-1)\geqslant \frac{3}{2}\mu\tau r\geqslant C_{17},$$
and $\tau\leqslant C_{16}h^{-1}$, we can get that for $h<h_1$,
\begin{equation*}
\|u\|_{B(p,r)\setminus B(p,\frac{r}{2})}\geqslant e^{-C_{18}/h}\|u\|_{B(p,2r)\setminus B(p,r)}.
\end{equation*}
Now \eqref{eq:dp} is a simple consequence of this estimate.
\end{proof}

\section{Nodal Sets for Solutions to Semiclassical Schr\"{o}dinger equations}
\label{sec:ap}
In this section, we assume $(M,g)$ is a real analytic compact Riemannian manifold, $V$ a real analytic function on $M$. Let $u=u(h)$ be the solution to the semiclassical Schr\"{o}dinger equation $(-h^2\Delta_g+V(x)-E(h))u=0$. We study the vanishing properties of $u$.

\subsection{Order of vanishing}
\label{sec:ov}
\begin{thm}
There exists a constant $C>0$ such that if $u$ vanishes at $x_0$ to the order $k$, then $k\leqslant Ch^{-1}$.
\end{thm}
\begin{proof}
Without loss of generality, we can assume $\|u(h)\|=1$. By Taylor's formula, for $|x-x_0|<\epsilon$,
\begin{equation}
|u(h)(x)|\leqslant\frac{\epsilon^k}{k!}\sup_{|\alpha|=k}\sup_{|y-x_0|<2\epsilon}|D^\alpha u(y)|
\end{equation}
Now we can apply semiclassical Cauchy estimates to get
\begin{equation}
|u(h)(x)|\leqslant(C\epsilon)^k(1+\frac{1}{hk})^k
\end{equation}
where $C$ is a constant only depending on $(M,g)$ and $V$ as long as $\epsilon$ is small enough.
If $k<\frac{1}{h}$, then there is nothing to prove. Otherwise, we can take $\epsilon$ small (not depend on $h$) to get
$$|u(h)(x)|\leqslant e^{-k}, |x-x_0|<\epsilon.$$
On the other hand, by Carleman estimate
\begin{equation}
\|u\|_{B(x_0,\epsilon)}\geqslant e^{-\frac{C}{h}}.
\end{equation}
Again we have $k\leqslant Ch^{-1}$.
\end{proof}

\begin{rem}
In \cite{B-C}, it is proved that this is true even when $(M,g)$ and $V$ are only smooth and the constant $C$ only depends on $(M,g)$ and the $C^1$-norm of $V$.
\end{rem}

\subsection{Nodal Set-Upper bounds}
\label{sec:nsu}
\begin{thm}
The $n-1$-dimension Hausdorff measure of the nodal set $N=\{x\in M,u(x)=0\}$ satisfies $\mathcal{H}^{n-1}(N)\leqslant Ch^{-1}$.
\end{thm}
\begin{proof}
From above, we know, for each $x\in M$, there exists $r=r(x)$ (independent of $h$) such that in local geodesic coordinates, $u$ can be analytic continued from $B(x,r)$ to $B_{\mathbb{C}^n}(x,\frac{r}{2})$ with $\sup\limits_{B_{\mathbb{C}^n}(x,\frac{r}{2})}|u|\leqslant e^{C/h}\sup\limits_{B(x,r)}|u|$.
We also recall the following technical lemma from \cite{D-F}.
\begin{lem}
There exists some constant $C>0$ such that if $H(z)$ is a holomorphic function in $|z|\leqslant2,z\in\mathbb{C}^n$ and for some $\alpha>1$
$$\sup_{B(x,\frac{1}{5})}|H|\geqslant e^{-\alpha}\max_{B_{\mathbb{C}^n}(0,2)}|H(z)|, x\in\mathbb{R}^n,|x|<\frac{1}{10}$$
then $H^{n-1}(\{x\in\mathbb{R}^n,|x|<\frac{1}{20},H(x)=0\})\leqslant C\alpha$.
\end{lem}
Now from this and the doubling property \eqref{eq:dp}, it is easy to see
\begin{equation*}
H^{n-1}(\{y\in B(x,\frac{r}{40}):u(y)=0\})\leqslant Ch^{-1}
\end{equation*}
By compactness, we can cover $M$ by finitely many $B(x,\frac{r}{40})$ and conclude the proof.
\end{proof}

\subsection{Nodal Set-Lower bounds in the classical allowed region}
\label{sec:nsl}

\begin{thm}
The $n-1$-dimensional Hausdorff measure of the nodal set $\mathcal{H}^{n-1}(N)\geqslant ch^{-1}$. More precisely, $\mathcal{H}^{n-1}(N_a)\geqslant ch^{-1}$, where $N_a=\{x\in M,u(x)=0,V(x)<E\}$ is the nodal set in the classical allowed region.
\end{thm}

First we prove a lemma
\begin{lem}
There exists $C>0$ such that for every small enough $h>0$, $u(h)$ has a zero in every ball $B(x_0,Ch)$ contained in the classical allowed region $\{x:V(x)<E\}$.
\end{lem}
\begin{proof}
We shall work in the normal geodesic coordinate centered at $x_0$. First, by comparison theorem, the first Dirichlet eigenvalue of $-h^2\Delta_g$ on $B(0,c_1h)$ is at most $c_2h$. Therefore the first Dirichlet eigenvalue of $-h^2\Delta_g+V(x)$ on $B(0,c_1h)$ is at most $V(0)+c_3h$. Let $u_0$ be the corresponding eigenfunction, we know $u_0>0$ on $B(0,c_1h)$. Now suppose $u>0$ on $B(0,c_1h)$, we set $w=u_0/u$, then $w=0$ on $\partial B(0,c_1h)$ and $w>0$ in $B(0,c_1h)$. Therefore $w$ achieves the maximum at some $x_1\in B(0,c_1h)$. We have at the point $x_1$,
$$0=\partial_iw=\frac{u\partial_iu_0-u_0\partial_iu}{u^2}$$
and
\begin{equation*}
\begin{split}
0\leqslant&\;-h^2\Delta_gw=\frac{(-h^2\Delta_gu_0)u-(-h^2\Delta_gu)u_0}{u^2}\\
=&\;\frac{(-h^2\Delta_g+V)u_0u-(-h^2\Delta_g+V)uu_0}{u^2}\leqslant\frac{(V(0)+c_3h-E(h))uu_0}{u^2}<0
\end{split}
\end{equation*}
if $h$ is small enough, since $V(0)<E$. This contradiction shows that $u$ must have a zero in the ball $B(x_0,Ch)$.
\end{proof}

The second lemma is a generalization of the mean-value formula for Euclidean Laplacian.
\begin{lem}
There exists $c>0$ and $c_0\in(0,1)$, such that for every $h\in(0,h_0)$, $p\in M$ such that $u(p)=0$,
$$\left|\int_{B(p,ch)}u\right|\leqslant c_0\int_{B(p,ch)}|u|$$
As a corollary, we have for some $c_1>0$,
$$\min\left\{\int_{B(p,ch)}u^+,\int_{B(p,ch)}u^-\right\}\geqslant c_1\int_{B(p,ch)}|u|$$
\end{lem}
\begin{proof}
The ideas of our proof is from \cite{C-M}. First, we define for any function $v$,
$$I_v(s)=s^{1-n}\int_{\partial B(p,s)}v.$$
Then
$$I_v'(s)=s^{1-n}\int_{B(p,s)}\Delta v+s^{1-n}\int_{\partial B(p,s)}v(\Delta d+\frac{1-n}{d}).$$
where $d(\cdot)=\dist(p,\cdot)$ denote the distance to the center of the ball. On $M$, when $d$ is small depending on $\sup|K_M|$ and the injective radius of $M$, by Hessian comparison theorem we have
$$|\Delta d+\frac{1-n}{d}|\leqslant k(d)$$
where $k:[0,\infty)\to\mathbb{R}$ is a continuous monotone non-degreasing function such that $k(0)=0$.

Let $l(s)=\max\limits_{x\in B(p,s)}|V(x)-V(p)|$, then $l:[0,\infty)$ is also a continuous monotone non-decreasing function such that $l(0)=0$. Moreover, $l(s)\leqslant(\max\limits_M|\nabla V|)s$.

Since
$$\int_{B(p,s)}u=\int_0^st^{n-1}I_u(t)dt,$$
we have
$$\left|\int_{B(p,s)}u\right|\leqslant\frac{s^n}{n}f(s)$$
where $f(s)=\max_{t\leqslant s}|I_u(t)|$. $f$ is also a continuous monotone non-decreasing function such that $f(0)=0$. Moreover $f$ is Lipschitz. Then
\begin{equation*}
\begin{split}
f'(s)\leqslant|I_u'(s)|\leqslant& h^{-2}s^{1-n}\left|\int_{B(p,s)}(V(x)-E)u\right|+s^{1-n}k(s)\int_{\partial B(p,s)}|u|\\
\leqslant&h^{-2}s^{1-n}|V(p)-E|\left|\int_{B(p,s)}u\right|
+h^{-2}s^{1-n}l(s)\int_{B(p,s)}|u|+k(s)I_{|u|}(s)\\
\leqslant&Ch^{-2}sf(s)+h^{-2}s^{1-n}l(s)\int_0^st^{n-1}I_{|u|}(t)dt
+k(s)I_{|u|}(s)\\
\leqslant&Ch^{-2}sf(s)+h^{-2}l(s)\int_0^sI_{|u|}(t)dt+k(s)I_{|u|}(s).
\end{split}
\end{equation*}
Therefore for $s\leqslant ch$,
$$f(s)\leqslant Ch^{-1}f(s)+Ch^{-1}\int_0^sI_{|u|}(t)dt+k(s)I_{|u|}(s).$$
By Gronwall's inequality, we have
\begin{equation*}
f(s)\leqslant e^{Ch^{-1}s}\int_0^s\left[Ch^{-1}\int_0^tI_{|u|}(t')dt'+k(t)I_{|u|}(t)\right]dt
\leqslant C\int_0^sI_{|u|}(t)dt.
\end{equation*}
Hence,
\begin{equation}
\left|\int_{B(p,ch)}u\right|\leqslant\frac{(ch)^n}{n}f(ch)
\leqslant Ch^n\int_0^{ch}I_{|u|}(t)dt.
\end{equation}
By rescaling $B(p,ch)$ to the ball of unit radius, we get a family of functions $\tilde{u}(x)=u(p+chx)$ on $B(0,1)$ solving a family of uniform elliptic equations $$\sum_{i,j=1}^na_{ij}\partial_i\partial_j\tilde{u}
+\sum_{i=1}^nb_i\partial_i\tilde{u}+c\tilde{u}=0.$$
Here
$$C^{-1}|\xi|^2\leqslant\sum_{i,j=1}^na_{ij}(x)\xi_i\xi_j\leqslant C|\xi|^2, |b_i(x)|\leqslant C, |c(x)|\leqslant C.$$
By standard elliptic estimate, (e.g. \cite{G-T}), we have
\begin{equation*}
\sup_{B(0,\frac{1}{2})}|\tilde{u}|\leqslant C\int_{B(0,1)}|\tilde{u}|.
\end{equation*}
for some $C>0$ uniformly in $\tilde{u}$. Back to $u$, we have
\begin{equation*}
\sup_{B(p,\frac{1}{2}ch)}|u|\leqslant Ch^{-n}\int_{B(p,ch)}|u|.
\end{equation*}
Therefore we have
\begin{equation*}
\begin{split}
\int_0^{ch}I_{|u|}(t)dt=&\int_0^{ch/2}\left(t^{1-n}\int_{\partial B(p,t)}|u|\right)dt+\int_{ch/2}^{ch}\left(t^{1-n}\int_{\partial B(p,t)}|u|\right)dt\\
\leqslant&\;\frac{ch}{2}\left(\sup_{B(p,\frac{1}{2}ch)}|u|\right)
\left(\sup_{t\leqslant ch/2}t^{1-n}\Vol(B(p,t))\right)
+\left(\frac{ch}{2}\right)^{1-n}\int_{ch/2}^{ch}\left(\int_{\partial B(p,t)}|u|\right)dt\\
\leqslant&\;Ch^{1-n}\int_{B(p,ch)}|u|.
\end{split}
\end{equation*}
The last inequality follows from the volume comparison theorem:
$$\sup\limits_{t\leqslant ch/2}t^{1-n}\Vol(B(p,t))$$
is bounded by a constant only depending on the curvature of $M$.
\end{proof}

Now following the idea of \cite{D-F}, we prove the theorem.
\begin{proof}
Assume $U\subset\subset\{V<E\}$ is a coordinate patch, then we can cover $U$ by cubes $Q_\nu$ of side $a_1h$ such that there exists a nodal point $x_\nu$ with $B_\nu=B(x_\nu,ch)\subset\subset\frac{1}{2}Q_\nu$(the cube with the same center and sides of half length of $Q_j$). Then by \cite[Proposition 5.11]{D-F} and the same argument as
\cite[Lemma 7.3,7.4]{D-F},
for at least half of the $Q_\nu$
$$\Av_{B_\nu}u^2\geqslant c\Av_{Q_\nu}u^2.$$
where $\Av_Uf$ denotes the average of $f$ on $U$: $\Av_Uf=|U|^{-1}\int_Uf$.
Now by standard elliptic theory, we have
$$\|u\|_{L^\infty(B_\nu)}\leqslant C(\Av_{Q_\nu}u^2)^{\frac{1}{2}}\leqslant C(\Av_{B_\nu}u^2)^{\frac{1}{2}}$$
Therefore
$$(\Av_{B_\nu}u^2)^{\frac{1}{2}}\leqslant C|B_\nu|^{-1}\int_{B_\nu}|u|$$
Let $B_\nu^+=\{x\in B_\nu:u>0\},B_\nu^-=\{x\in B_\nu,u<0\}$, then
$$\int_{B_\nu^\pm}u\leqslant\|u\|_{B_\mu^{\pm}}|B_\mu^{\pm}|^{\frac{1}{2}}\leqslant
C(\frac{|B_\mu^{\pm}|}{|B_\mu|})^{\frac{1}{2}}(\Av_{B_\nu}u^2)^{\frac{1}{2}}$$
Combining this with the previous lemma, we have
$$\min\{|B_\mu^+|,|B_\mu^-|\}\geqslant c|B_\mu|.$$
The isoperimetric inequality shows that
$$H^{n-1}(B_\mu\cap N)\geqslant c\min\{|B_\mu^+|,|B_\mu^-|\}^{\frac{n-1}{n}}\geqslant ch^{n-1}$$
Since we have at least $ch^{-n}$ such cubes $Q_\nu$, we conclude that
$$H^{n-1}(N\cap U)\geqslant ch^{-1}$$
\end{proof}

\end{document}